\documentclass[a4paper,12pt,reqno]{amsart}

\usepackage{amsmath,amssymb,amsthm}
\usepackage{latexsym}
\usepackage{color}
\usepackage{graphicx}
\usepackage{mathrsfs}
\usepackage{enumerate}
\usepackage[abbrev]{amsrefs}
\usepackage[T1]{fontenc}
\usepackage{mathtools}
\mathtoolsset{showonlyrefs=true}

\allowdisplaybreaks[4]

\theoremstyle{plain}
\newtheorem{thm}{Theorem}[section]
\newtheorem{lemm}[thm]{Lemma}

\theoremstyle{definition}

\newtheorem{rem}[thm]{Remark}

\makeatletter

\@addtoreset{equation}{section}
\makeatother

\renewcommand{\div}{\operatorname{div}}
\newcommand{\dB}{\dot{B}}
\newcommand{\supp}{\operatorname{supp}}

\renewcommand{\leq}{\leqslant}
\renewcommand{\geq}{\geqslant}

\newcommand{\uw}{u^{\rm w}}
\newcommand{\pw}{p^{\rm w}}
\newcommand{\Fw}{F^{\rm w}}

\newcommand{\n}[1]{{\left\|#1\right\|}}

\begin{document}
\title[Stationary Navier--Stokes equations on the half spaces]
{Stationary Navier--Stokes equations on the half spaces in the scaling critical framework}
\author[M.Fujii]{Mikihiro Fujii}
\address[M.Fujii]{Graduate School of Science, Nagoya City University, Nagoya, 467-8501, Japan}
\email[M.Fujii]{fujii.mikihiro@nsc.nagoya-cu.ac.jp}
\begin{abstract}
In this paper, we consider the inhomogeneous Dirichlet boundary value problem for the stationary Navier--Stokes equations in $n$-dimensional half spaces $\mathbb{R}^n_+= \{ x=(x',x_n)\ ;\ x' \in \mathbb{R}^{n-1}, x_n > 0 \}$ with $n \geq 3$
and prove the well-posedness\footnote{This means the unique existence of the small solutions for given small data.} in the scaling critical Besov spaces.
Our approach is to regard the system as an evolution equation for the normal variable $x_n$ and reformulate it as an integral equation.
Then, we achieve the goal by making use of the maximal regularity method that has developed in the context of nonstationary analysis in critical Besov spaces.
Furthermore, for the case of $n \geq 4$, we find that the asymptotic profile of the solution as $x_n \to \infty$ is given by the $(n-1)$-dimensional stationary Navier--Stokes flow.
\end{abstract}
\maketitle

\section{Introduction}\label{sec:intro}
In this paper, we consider the incompressible stationary Navier--Stokes equations in the half space $\mathbb{R}^n_+:= \{ x=(x',x_n)\ ;\ x'=(x_1,...,x_{n-1}) \in \mathbb{R}^{n-1}, x_n > 0 \}$ for $n \geq 3$ with the inhomogeneous Dirichlet boundary condition:  
\begin{align}\label{eq:NS}
    \begin{cases}
        -\Delta u + (u \cdot \nabla) u + \nabla p = \div F, & \qquad x \in \mathbb{R}^n_+,\\
        \div u = 0, &\qquad x \in \mathbb{R}^n_+,\\
        u(x',0) = a(x'), &\qquad x' \in \mathbb{R}^{n-1},
    \end{cases}
\end{align}
where $u=(u',u_{n})=(u_{1}(x),...,u_{n-1}(x),u_{n}(x))$ and $p=p(x)$ represent the unknown velocity field and pressure of the fluid, while the external force is assumed to be the divergence form $\div F$ for some given function $F=\{F_{k,\ell}(x) \}_{1 \leq k,\ell \leq n}$, and $a=(a',a_n)=(a_1(x'),...,a_{n-1}(x'),a_n(x'))$ denotes the given boundary data.
The aim of this paper is to show the existence of a unique small solution to \eqref{eq:NS} in the scaling critical framework.
Moreover, for the case of $n \geq 4$ and the external force $F$ is independent of the normal variable $x_n$, then it is revealed that the solution $u=u(x',x_n)$ converges to the solution $\bar{u}=\bar{u}(x')$ of the modified $(n-1)$-dimensional stationary Navier--Stokes equations \eqref{eq:NS_n-1} below.

For the whole space case $\mathbb{R}^n$ with $n \geq 3$, the stationary Navier--Stokes equations
\begin{align}\label{eq:NS-w}
    \begin{cases}
        -\Delta u + (u \cdot \nabla) u + \nabla p = \div F, & \qquad x \in \mathbb{R}^n,\\
        \div u = 0, &\qquad x \in \mathbb{R}^n
    \end{cases}
\end{align}
are well-investigated and
a number of results are known. The pioneering well-posedness results are given by
Leray \cite{Ler-33}, Ladyzhenskaya \cite{Lad-59}, and Fujita \cite{Fuj-61}.
The system \eqref{eq:NS-w} has the scaling invariant structure, that is if $u$ and $p$ solve \eqref{eq:NS-w} with some external force $F$, then the scaled functions 
\begin{align}\label{scaling-u-p}
    u_{\lambda}(x):=\lambda u(\lambda x),\qquad p_{\lambda}(x):=\lambda^2p(\lambda x)
\end{align}
also solve \eqref{eq:NS-w} with the scaled external force 
\begin{align}\label{scaling-F}
    F_{\lambda}(x):=\lambda^2F(\lambda x)
\end{align}
for all $\lambda >0$.
It is said that a function space is scaling critical if its norm is invariant under the above scaling transforms \eqref{scaling-u-p}-\eqref{scaling-F}.
It is well-known as the Fujita--Kato principle (see \cite{Fuj-Kat-64}) that it is important to consider the solvability of partial differential equations in the scaling critical function spaces.
For the results on \eqref{eq:NS-w} in the scaling critical spaces framework,
Chen \cite{Che-93} proved the well-posedness of \eqref{eq:NS-w} from small $F\in L^{\frac{n}{2}}(\mathbb{R}^n)$ to small $u \in L^n(\mathbb{R}^n)$.
Kaneko--Kozono--Shimizu \cite{Kan-Koz-Shi-19} proved that \eqref{eq:NS-w} is well-posed\footnote{In this paper, well-posedness means the unique existence of small solutions for given small data and the continuous dependence.} from small $F \in \dB_{p,r}^{\frac{n}{p}-2}(\mathbb{R}^n)$ to small $ u \in \dB_{p,r}^{\frac{n}{p}-1}(\mathbb{R}^n)$ for all $1 \leq p < n$ and $1 \leq r \leq \infty$, 
whereas
Tsurumi \cites{Tsu-19-JMAA,Tsu-19-ARMA} showed the ill-posedness for the case $p \geq n$.
For other related results,
see 
Kozono--Yamazaki \cites{Koz-Yam-95-PJA, Koz-Yam-95-IUMJ} for the well-posedness and stability in the scaling critical Morrey spaces,
Tsurumi \cite{Tsu-19-DIE} for the well-posedness in the scaling critical Triebel--Lizorkin spaces, 
and 
Heywood \cite{Hey-70}, Kozono--Shimizu \cite{Koz-Shi-23}, and Cunanan--Okabe--Tsutsui \cite{Cun-Oka-Tsu-22} for the large-time asymptotic stability for nonstationary flow around the stationary solutions.
Next, we focus on the known results of the nonstationary Navier--Stokes equations on the half spaces $\mathbb{R}^n_+$ ($n \geq 2$):
\begin{align}\label{eq:nNS}
    \begin{cases}
        \partial_t u - \Delta u + (u \cdot \nabla) u + \nabla p = 0, \qquad & t > 0, x \in \mathbb{R}^n_+,\\
        \div u = 0 \qquad & t \geq 0, x \in \mathbb{R}^n_+,\\
        u(t, x',0) = a(t,x'), \qquad & t \geq 0, x' \in \mathbb{R}^{n-1},\\
        u(0,x) = u_0(x), \qquad & x \in \mathbb{R}^n_+.
    \end{cases}
\end{align}
Since the pioneering work of Ukai \cite{Uka-87}, where the explicit formula for the linear solution is derived, the solvability of \eqref{eq:nNS} is well-investigated.
In \cite{Koz-87}, the well-posedness for \eqref{eq:nNS} with $a=0$ in the scaling critical space $L^n(\mathbb{R}^n_+)$ is proved.
Yamazaki \cite{Wat-24} extended this results to the wider class $L^{n,\infty}(\mathbb{R}^n_+)$. 
Recently, Watanabe \cite{Wat-24} considered the maximal regularity of the Stokes semigroup on the half space in homogeneous Besov spaces to prove that \eqref{eq:nNS} with $a=0$ is global well-posed in the scaling critical Besov spaces $\dB_{p,1}^{\frac{n}{p}-1}(\mathbb{R}^n_+)$ with $n-1 < p < \infty$.
For the results of inhomogeneous boundary data case $a \neq 0$, 
Lewis \cite{Lew-72} and Voss \cite{Vos-96} proved the global well-posedness for small data in the Lebesgue spaces framework.
In \cite{Fer-Fer-13}, it is shown that \eqref{eq:nNS} is global well-posed in the Morrey space $M_{p,n-p}(\mathbb{R}^n_+)$ with ($2 < p \leq n$), which is a wider class than $L^{n,\infty}(\mathbb{R}^n_+)$. 
More precisely, it was shown that for exponents $p$, $q$, and $r$ satisfying $2<p<q/p'<r<q<\infty$ and given initial data $u_0 \in M_{p,n-p}(\mathbb{R}^n_+)$ and boundary data $a \in C((0,\infty);M_{r,n-p}(\mathbb{R}^{n-1}) \cap M_{\frac{q}{p'},n-p}(\mathbb{R}^{n-1}))$ satisfying 
\begin{align}
    \n{u_0}_{M_{p,n-p}}
    \ll 1,
    \quad
    \sup_{t>0}
    t^{\frac{1}{2}-\frac{p-1}{2r}}
    \n{a(t)}_{M_{r,n-p}}
    +
    \sup_{t>0}
    t^{\frac{1}{2}-\frac{p}{2r}}
    \n{a(t)}_{M_{\frac{q}{p'},n-p}}
    \ll 1,
\end{align} 
there exists a unique small global solution.
Moreover, Chang and Jin \cites{Cha-Jin-15,Cha-Jin-16} proved the local and global well-posedness in the class $u_0 \in \dB_{q,q}^{-\frac{2}{q}}(\mathbb{R}^n_+)$ and $a \in L^q(0,T;\dB_{q,q}^{-\frac{1}{q}}(\mathbb{R}^{n-1}))+L^q(\mathbb{R}^{n-1};\dB_{q,q}^{-\frac{1}{2q}}(0,T))$ with $a_n \in L^q(0,T;\dB_{q,q}^{-\frac{1}{q}}(\mathbb{R}^{n-1}))$ for $n+2<q<\infty$.

In contrast to the known results mentioned above, there are few studies of stationary Navier--Stokes equations \eqref{eq:NS} in the half space cases, and in particular, there seems no solvability result for critical spaces in $\mathbb{R}^{n}_+$.
The aim of this paper is to consider the inhomogeneous Dirichlet boundary problem of the stationary Navier--Stokes equations \eqref{eq:NS} and show the well-posedness in critical Besov spaces framework.
Here, we note that the scaling transforms \eqref{scaling-u-p} and \eqref{scaling-F} are also invariant for the system \eqref{eq:NS} with the scaled Dirichlet data:
\begin{align}\label{scaling-a}
    a_{\lambda}(x):= \lambda a(\lambda x).
\end{align}
To achieve this, we regard \eqref{eq:NS} as an evolution equation for the vertical variable $x_n>0$ with regarding the boundary data $a=a(x')$ as the initial data. 
Then, we reformulate it into the integral equation via the Fourier transform for the tangential variables and the formula $\div F - (u \cdot \nabla )u = \div (F - u \otimes u)$:
\begin{align}\label{eq:int-1}
    u
    =
    \mathcal{U}^{\rm boundary}[a] + \mathcal{U}^{\rm force}[F - u \otimes u],
\end{align}
where 
$\mathcal{U}^{\rm boundary}[a](x)$ consists of the terms like $e^{-x_n|\nabla'|}a(x')$ 
and 
$\mathcal{U}^{\rm force}[G](x)$ is composed of the terms like $\int_0^{\infty} e^{-|x_n- y_n||\nabla'|}G(x',y_n) dy_n$.
As the actual representation of $\mathcal{U}^{\rm boundary}[a]$ and $\mathcal{U}^{\rm force}[G]$ are complicated, see Section \ref{sec:lin} for the detail. 
Thus, the integral equation \eqref{eq:int-1} is similar to the following integral equation:
\begin{align}
    v(x',x_n) 
    ={}
    e^{-x_n|\nabla'|}a(x')
    + 
    \int_0^{\infty} e^{-|x_n-y_n||\nabla'|}(f-v^2)(x',y_n)dy_n.
\end{align}
We establish the maximal regularity estimates of the semi-group $\{ e^{-x_n|\nabla'|} \}_{x_n>0}$ and the bilinear estimates in Besov spaces, so that we may construct a unique small solution to \eqref{eq:NS} in the scaling critical class
$u \in \widetilde{L^{\infty}_{x_n}}(\dB_{p,r}^{\frac{n-1}{p}-1})_{x'}(\mathbb{R}^n_+)$
for given small boundary data $a \in \dB_{p,r}^{\frac{n-1}{p}-1}(\mathbb{R}^{n-1})$ and small external force $F \in \widetilde{L^q_{x_n}}(\dB_{p,r}^{\frac{n-1}{p}+\frac{1}{q}-2})_{x'}(\mathbb{R}^{n}_+)$ for some $1 \leq p,q,r \leq \infty$. See Section \ref{sec:main} for the detail.
Our approach is different from that of Watanabe \cite{Wat-24} and enables us to treat the non-zero boundary data and consider the solutions in frameworks of different integrability exponents for $x'$-direction and $x_n$-direction. 
In particular, we obtain solutions that may not decay at $x_n \to \infty$.
Furthermore, we also consider the asymptotic behavior of the solution $u$ of \eqref{eq:NS} as $x_n \to \infty$ when the external force is independent of the normal variable $x_n>0$ and prove that $u$ approaches to the solution to $(n-1)$-dimensional stationary Navier--Stokes equations \eqref{eq:NS_n-1} in the scaling critical $\dB_{p,r}^{\frac{n-1}{p}-1}(\mathbb{R}^{n-1})$-norm.

Throughout this paper, we denote by $C$ and $c$ the constants, which may differ in each line. In particular, $C=C(*,...,*)$ denotes the constant which depends only on the quantities appearing in parentheses. 
For a vector $x =(x_1,...,x_n) \in \mathbb{R}^n$, we write $x':=(x_1,...,x_{n-1})$.
For an integrability exponent $1 \leq q \leq \infty$, we denote by $q':=q/(q-1)$ the H\"older conjugate of $q$.

This paper is organized as follows.
In the next section, we prepare some notation and state our main results precisely.
In Section \ref{sec:lin}, we derive the solution formula for the linear equations and establish some linear estimates.
In Section \ref{sec:pf}, we provide nonlinear estimates, so that we combine them and the results established in Section \ref{sec:lin} to prove our main theorems.

\section{Main results}\label{sec:main}
In this section, we state main results of this paper.
We first prepare some notations.
Let $n,d \in \mathbb{N}$ satisfy $n \geq 2$ and $1 \leq d \leq n$.
We denote by $\mathscr{S}(\mathbb{R}^d)$ the set of all Schwartz functions on $\mathbb{R}^d$ and define $\mathscr{S}'(\mathbb{R}^d)$ as the set of all tempered distributions on $\mathbb{R}^d$.
Let $\mathscr{F}_{\mathbb{R}^d}$ and $\mathscr{F}^{-1}_{\mathbb{R}^d}$ be the Fourier transform and inverse Fourier transform, respectively, defined as follows:
\begin{align}
    \mathscr{F}_{\mathbb{R}^d} [ f ](\xi) = \int_{\mathbb{R}^d} e^{-ix\cdot \xi}f(x) dx,\qquad
    \mathscr{F}_{\mathbb{R}^d}^{-1} [ f ](x) = \frac{1}{(2\pi)^d} \int_{\mathbb{R}^d} e^{ix\cdot \xi}f(\xi) d\xi
\end{align}
for $f \in \mathscr{S}(\mathbb{R}^d)$.
We use the abbreviation $\widehat{f}(\xi')=\mathscr{F}_{\mathbb{R}^{n-1}}[f](\xi')$ only for the case $d=n-1$.
Next, we recall the Littlewood--Paley decomposition $\{\Delta_j\}_{j\in \mathbb{Z}}$ on $\mathbb{R}^{n-1}$ defined by
\begin{align}
    \Delta_j f := \varphi_j * f, \qquad \varphi_j(x')=2^{(n-1)j}\varphi_0(2^jx'),
\end{align}
where $*$ denotes the convolution on $\mathbb{R}^{n-1}$ and 
$\varphi_0 \in \mathscr{S}(\mathbb{R}^{n-1})$ satisfies 
\begin{align}
    0 \leq \widehat{\varphi_0}(\xi') \leq 1, \quad
    \supp \widehat{\varphi_0} \subset \{ \xi' \in \mathbb{R}^{n-1}\ ;\ 2^{-1} \leq |\xi'| \leq 2 \},
\end{align}
and
\begin{align}
    \sum_{j \in \mathbb{Z}} \widehat{\varphi_0}(2^{-j}\xi') = 1\qquad {\rm for\ all\ }\xi' \in \mathbb{R}^{n-1} \setminus \{ 0 \}.
\end{align}
Then, the Besov spaces $\dB_{p,r}^s(\mathbb{R}^{n-1})$ ($s \in \mathbb{R}$, $1 \leq p,r \leq \infty$) are defined as 
\begin{align}
    \dB_{p,r}^s(\mathbb{R}^{n-1})
    :={}&
    \left\{
    f \in \mathscr{S}'(\mathbb{R}^{n-1})/\mathscr{P}(\mathbb{R}^{n-1})\ ;\ 
    \n{f}_{\dB_{p,r}^s} < \infty
    \right\},\\
    \n{f}_{\dB_{p,r}^s}
    :={}&
    \n{\left\{ 2^{sj}\n{\Delta_jf}_{L^p_{x'}(\mathbb{R}^{n-1})} \right\}_{j \in \mathbb{Z}}}_{\ell^{r}(\mathbb{Z})},
\end{align}
where $\mathscr{P}(\mathbb{R}^{n-1})$ denotes the set of all polynomials on $\mathbb{R}^{n-1}$.
See Sawano \cite{Saw-18} for the basic properties of Besov spaces.
Since we regard \eqref{eq:NS} as an evolution equation for $x_n \in (0,\infty)$ with the initial data $a=a(x')$,
we use the Chemin--Lerner spaces $\widetilde{L^q_{x_n}}(a,b;(\dB_{p,r}^{s})_{x'}(\mathbb{R}^{n-1}))$ ($0\leq a<b\leq \infty$, $1\leq p,q,r \leq \infty$, and $s \in \mathbb{R}$) defined as 
\begin{align}
    &\widetilde{L^q_{x_n}}(a,b;(\dB_{p,r}^{s})_{x'}(\mathbb{R}^{n-1}))\\
    &\qquad
    :={}
    \left\{
    f : (a,b) \to \mathscr{S}'(\mathbb{R}^{n-1})/\mathscr{P}(\mathbb{R}^{n-1})\ ;\ 
    \n{f}_{\widetilde{L^q_{x_n}}(a,b;(\dB_{p,r}^{s})_{x'})}
    < 
    \infty
    \right\},\\
    &\n{f}_{\widetilde{L^q_{x_n}}(a,b;(\dB_{p,r}^{s})_{x'})}
    :={}
    \n{\left\{ 2^{sj}\n{\Delta_jf}_{L^q_{x_n}(a,b;L^p_{x'}(\mathbb{R}^{n-1}))} \right\}_{j \in \mathbb{Z}}}_{\ell^{r}(\mathbb{Z})}.
\end{align}
Chemin--Lerner spaces were first introduced in \cite{Che-Ler-95} and are nowadays frequently used in the analysis of nonstationary viscous compressible fluid in scaling critical Besov spaces framework, which is started by Danchin \cite{Dan-00}. 
The reason why we use the Chemin--Lerner space is that this space enables us to obtain the maximal regularity estimates for evolution equations (see Lemma \ref{lemm:max-reg}).  
See \cite{Bah-Che-Dan-11} for properties of Chemin--Lerner spaces. 
We use the abbreviation 
\begin{align}
    \widetilde{L^q_{x_n}}(\dB_{p,r}^{s})_{x'}(\mathbb{R}^{n}_+):=
    \widetilde{L^q_{x_n}}(0,\infty;(\dB_{p,r}^{s})_{x'}(\mathbb{R}^{n-1})).
\end{align}

Now, the first main result of this paper reads as follows:
\begin{thm}\label{thm:1}
    Let $n\geq3$ be an integer.
    Let $1 \leq q,r\leq \infty$ and let
    \begin{align}
        1 \leq p <q_*'(n-1), 
        \qquad 
        q_*:=\max\{2,q\}, 
    \end{align}
    where $q_*':=q_*/(q_*-1)$ denotes the H\"older conjugate of $q_*$.
    We additionally assume that $q < \infty$ if $n=3$.
    Then, there exist positive constants $\delta_0=\delta_0(p,q,r)$ and $\varepsilon_0=\varepsilon_0(p,q,r)$ such that 
    for any boundary data
    $a \in \dB_{p,r}^{\frac{n-1}{p}-1}(\mathbb{R}^{n-1})$ and external forces
    $F \in \widetilde{L^q_{x_n}}(\dB_{p,r}^{\frac{n-1}{p}+\frac{1}{q}-2})_{x'}(\mathbb{R}^{n}_+)$
    satisfying 
    \begin{align}\label{small}
        \| a \|_{\dB_{p,r}^{\frac{n-1}{p}-1}}
        +
        \| F \|_{\widetilde{L^q_{x_n}}(\dB_{p,r}^{\frac{n-1}{p}+\frac{1}{q}-2})_{x'}}\leq \delta_0,
    \end{align}
    \eqref{eq:NS} possesses a solution $u 
    \in
    {} 
    \widetilde{L^{\infty}_{x_n}}(\dB_{p,r}^{\frac{n-1}{p}-1})_{x'}(\mathbb{R}^n_+) 
    \cap \widetilde{L^{q_*}_{x_n}}(\dB_{p,r}^{\frac{n-1}{p}+\frac{1}{q_*}-1})_{x'}(\mathbb{R}^n_+)$ satisfying 
    $\| u \|_{\widetilde{L^{q_*}_{x_n}}(\dB_{p,r}^{\frac{n-1}{p}+\frac{1}{q_*}-1})_{x'}} \leq \varepsilon_0.$
    Moreover, the solution is unique in the class
    \begin{align}\label{thm:class}
        \left\{
        u 
        \in
        {} 
        \widetilde{L^{q_*}_{x_n}}(\dB_{p,r}^{\frac{n-1}{p}+\frac{1}{q_*}-1})_{x'}(\mathbb{R}^n_+)\ ;\ 
        \| u \|_{\widetilde{L^{q_*}_{x_n}}(\dB_{p,r}^{\frac{n-1}{p}+\frac{1}{q_*}-1})_{x'}} \leq \varepsilon_0
        \right\}.
    \end{align}
\end{thm}

\begin{rem}
We mention some remarks on Theorem \ref{thm:1}.
\begin{enumerate}
    \item 
    Considering the invariant scalings \eqref{scaling-u-p}, \eqref{scaling-F}, and \eqref{scaling-a}, we see that 
    \begin{align}
        \n{a_{\lambda}}_{\dB_{p,r}^{\frac{n-1}{p}-1}} 
        &= \n{a}_{\dB_{p,r}^{\frac{n-1}{p}-1}},\\
        \| F_{\lambda} \|_{\widetilde{L^q_{x_n}}(\dB_{p,r}^{\frac{n-1}{p}+\frac{1}{q}-2})_{x'}}
        &=
        \| F \|_{\widetilde{L^q_{x_n}}(\dB_{p,r}^{\frac{n-1}{p}+\frac{1}{q}-2})_{x'}},
    \end{align}
    and
    \begin{align}
        \| u_{\lambda} \|_{\widetilde{L^{\infty}_{x_n}}(\dB_{p,r}^{\frac{n-1}{p}-1})_{x'}}
        &=
        \| u \|_{\widetilde{L^{\infty}_{x_n}}(\dB_{p,r}^{\frac{n-1}{p}-1})_{x'}}\\
        \| u_{\lambda} \|_{\widetilde{L^{q_*}_{x_n}}(\dB_{p,r}^{\frac{n-1}{p}+\frac{1}{q_*}-1})_{x'}}
        &=
        \| u \|_{\widetilde{L^{q_*}_{x_n}}(\dB_{p,r}^{\frac{n-1}{p}+\frac{1}{q_*}-1})_{x'}}
    \end{align}
    for all dyadic numbers $\lambda>0$. Thus, our framework is scaling critical.
    \item Next, 
    let us compare Theorem \ref{thm:1} with the well-posedness results on the whole space case.
    Kaneko--Kozono--Shimizu \cite{Kan-Koz-Shi-19} proved that for $1 \leq p < n$ and $1 \leq r \leq \infty$ small external force $F \in \dB_{p,r}^{\frac{n}{p}-2}(\mathbb{R}^n)$, there exists a small unique solution $u \in \dB_{p,r}^{\frac{n}{p}-1}(\mathbb{R}^n)$ to \eqref{eq:NS-w}, 
    while 
    Tsurumi \cites{Tsu-19-JMAA,Tsu-19-ARMA} implies the range $1 \leq p < n$ is optimal. 
    In contrast, our result considers the different integrable exponents for $x'$-direction and $x_n$-direction, which enable us to construct solutions with partially weaker decay than the  solutions in \cite{Kan-Koz-Shi-19}. 
    In particular, if we may choose $q=\infty$ with $n \geq 4$, then the external force $F$ and the solution $u$ may not decay as $x_n \to \infty$.
\end{enumerate}
\end{rem}
In the next theorem, we investigate the behavior of the solution constructed in Theorem \ref{thm:1} in the case $q=\infty$ and deduce the asymptotic profile of the non-decaying solution as $x_n \to \infty$.
\begin{thm}\label{thm:2}
    Let $n \geq 4$ be an integer and let $p$ and $r$ satisfy
    \begin{align}
        1 \leq p < n-1, 
        \quad 
        1 \leq r < \infty.
    \end{align}
    Then, there exist positive constants $\delta_1=\delta_1(n,p,r) \leq \delta_0$ and $\varepsilon_1=\varepsilon_1(n,p,r) \leq \varepsilon_0$ such that 
    for the unique solution $u=u(x',x_n)$ to \eqref{eq:NS} constructed in Theorem \ref{thm:1}
    with the boundary data
    $a \in \dB_{p,r}^{\frac{n-1}{p}-1}(\mathbb{R}^{n-1})$ and the external force $\bar{F} = \{ \bar{F}_{k,\ell}(x') \}_{1 \leq k,\ell \leq n} \in \dB_{p,r}^{\frac{n-1}{p}-2}(\mathbb{R}^{n-1})$
    satisfying the smallness condition \eqref{small} with $\delta_0$ replaced by $\delta_1$,
    it holds  
    \begin{align}
        \lim_{R \to \infty} \n{u - \bar{u}}_{\widetilde{L^{\infty}_{x_n}}(R,\infty;(\dB_{p,r}^{\frac{n-1}{p}-1})_{x'})} = 0,    
    \end{align}
    where $\bar{u}=(\bar{u}',\bar{u}_n)=(\bar{u}_1(x'),...,\bar{u}_{n-1}(x'),\bar{u}_n(x')) \in \dB_{p,r}^{\frac{n-1}{p}-1}(\mathbb{R}^{n-1})$ is a unique solution to the following modified $(n-1)$-dimensional stationary Navier--Stokes equations
    \begin{align}\label{eq:NS_n-1}
        \begin{cases}
            -\Delta' \bar{u}' + (\bar{u}' \cdot \nabla')\bar{u}' + \nabla' \bar{p} =  \div' \bar{F}', \qquad & x' \in \mathbb{R}^{n-1},\\
            -\Delta' \bar{u}_n + (\bar{u}' \cdot \nabla')\bar{u}_n  =  \div' \bar{F}_n, \qquad & x' \in \mathbb{R}^{n-1},\\
            \div' \bar{u}' = 0, & x' \in \mathbb{R}^{n-1},
        \end{cases}
    \end{align}
    satisfying $\| \bar{u} \|_{\dB_{p,r}^{\frac{n-1}{p}-1}} \leq \varepsilon_1$, where we have set $\bar{F}':=\{ \bar{F}_{k,\ell} \}_{1 \leq k,\ell \leq n-1}$ and $\bar{F}_n:= (\bar{F}_{n,1},...,\bar{F}_{n,n-1})$. 
    Here, $\Delta'$, $\nabla'$, and $\div'$ are the $(n-1)$-dimensional Laplacian, gradient, and divergence, respectively, and $\bar{p} = \bar{p}(x')$ denotes the pressure in $\mathbb{R}^{n-1}$.
\end{thm}
\begin{rem}
Let us state some remarks on Theorem \ref{thm:2}.
\begin{itemize}
    \item [(1)]
    For the existence of a unique solution to \eqref{eq:NS_n-1}, it follows from the same arguments in \cite{Kan-Koz-Shi-19} that for $1 \leq p < n-1$, $1 \leq r \leq \infty$, and $n-1 \geq 3$ there exist positive constants $\delta_1$ and $\varepsilon_1$ such that for any $\bar{F} \in \dB_{p,r}^{\frac{n-1}{p}-2}(\mathbb{R}^{n-1})$ with $\| \bar{F} \|_{\dB_{p,r}^{\frac{n-1}{p}-2}} \leq \delta_1$, \eqref{eq:NS_n-1} possesses a unique solution $\bar{u} \in \dB_{p,r}^{\frac{n-1}{p}-1}(\mathbb{R}^{n-1})$ satisfying $\| \bar{u} \|_{\dB_{p,r}^{\frac{n-1}{p}-1}} \leq \varepsilon_1$.
    Here, we remark that in order to construct a unique solution of the stationary equations \eqref{eq:NS_n-1} by following the idea of \cite{Kan-Koz-Shi-19}, we need to assume $n -1 \geq 3$, that is $n \geq 4$.
    \item [(2)]
    For the case of $n=3$, the limit system \eqref{eq:NS_n-1} are the equations on the whole plane $\mathbb{R}^2$ and the author \cite{Fujii-pre}*{Corollary 1.4} proved that there exists a small external force that does not generate the small solution in $\dB_{p,1}^{\frac{2}{p}-1}(\mathbb{R}^2)$ ($1\leq p \leq 2$) to the stationary Navier--Stokes equations on the whole plane $\mathbb{R}^2$.
Therefore, \cite{Fujii-pre} implies that Theorem \ref{thm:2} may not hold for $n=3$.
\end{itemize}
\end{rem}

\section{Linear analysis}\label{sec:lin}
In this section, we derive the explicit formula and estimate for the solutions to the following linear equations:
\begin{align}\label{eq:lin-1}
    \begin{cases}
        -\Delta u + \nabla p = \div F, & \qquad x \in \mathbb{R}^n_+,\\
        \div u = 0, &\qquad x \in \mathbb{R}^n_+,\\
        u(x',0) = a(x'), &\qquad x' \in \mathbb{R}^{n-1}.
    \end{cases}
\end{align}
Let $\Fw = \{ \Fw_{k,\ell} \}_{1 \leq k,\ell \leq n}$ be a extension of $F$ to the whole spaces defined as 
\begin{align}
    \Fw_{k,\ell}(x)
    :={}&
    \begin{cases}
        F_{k,\ell}(x',x_n), & x_n \geq 0,\\
        F_{k,\ell}(x',-x_n), & x_n < 0,
    \end{cases}\\
    \Fw_{k,n}(x)
    :={}&
    \begin{cases}
        F_{k,n}(x',x_n), & x_n \geq 0,\\
        -F_{k,n}(x',-x_n), & x_n < 0
    \end{cases}
\end{align}
for $k=1,...,n$ and $\ell=1,...,n-1$.
Let $(\uw, \pw)$ satisfy
\begin{align}\label{eq:lin-2}
    \begin{cases}
        -\Delta \uw + \nabla \pw = \div \Fw, & \qquad x \in \mathbb{R}^n,\\
        \div \uw = 0, &\qquad x \in \mathbb{R}^n.
    \end{cases}
\end{align}
Let $v:=u-\uw$ and $q:=p-\pw$.
Then, $(v,q)$ should solve
\begin{align}\label{eq:lin-3}
    \begin{cases}
        -\Delta v + \nabla q = 0, & \qquad x \in \mathbb{R}^n_+,\\
        \div v = 0, &\qquad x \in \mathbb{R}^n_+,\\
        v(x',0) =b(x'), &\qquad x' \in \mathbb{R}^{n-1},
    \end{cases}
\end{align}
where $b(x'):=a(x') - \uw(x',0)$.
To obtain the explicit formula of the solutions to \eqref{eq:lin-3}, we use the idea of \cite{Uka-87}.
Applying $\div$ to the first equation of \eqref{eq:lin-3}, we see that $\Delta q = 0$, which implies 
$(|\xi'|^2 - \partial_{x_n}^2)\widehat{q} = 0$.
As $\widehat{q}(\cdot,x_n) \in \mathscr{S}'(\mathbb{R}^{n-1})$ for all $x_n >0$, we have
\begin{align}\label{q}
    (|\xi'| + \partial_{x_n})\widehat{q}=0.
\end{align}
Let $\widehat{w_n}(\xi',x_n):=(|\xi'| + \partial_{x_n})\widehat{v_n}(\xi',x_n) = |\xi'|\widehat{v_n}(\xi',x_n) - i\xi'\cdot \widehat{v'}(\xi',x_n)$.
Then, it holds
\begin{align}
    (|\xi'| - \partial_{x_n})\widehat{w_n} + \partial_{x_n} \widehat{q} = 0. 
\end{align}
Applying $(|\xi'| + \partial_{x_n})$ to this and using \eqref{q}, we see that 
\begin{align}
    (|\xi'|^2 - \partial_{x_n}^2)\widehat{w_n} = 0.
\end{align}
By $\widehat{w_n}(\cdot,x_n) \in \mathscr{S}'(\mathbb{R}^{n-1})$, it holds
\begin{align}
    \begin{cases}
        (|\xi'| + \partial_{x_n})\widehat{w_n} = 0, &\qquad x_n>0,\\
        \widehat{w_n}(\xi',0) = |\xi'|\widehat{b_n}(\xi') - i \xi' \cdot \widehat{b'}(\xi'), &\qquad x_n=0.
    \end{cases}
\end{align}
Solving this equation, we have $\widehat{w_n}(\xi',x_n) = e^{-x_n|\xi'|}|\xi'|\widehat{b_n}(\xi')$, which gives
\begin{align}\label{v_n}
    \widehat{v_n}(\xi',x_n) 
    = 
    e^{-x_n|\xi'|}
    \left\{
    ( 1 + x_n|\xi'|)\widehat{b_n}(\xi')
    -ix_n\xi'\cdot \widehat{b'}(\xi')
    \right\}.
\end{align}
Let $\widehat{w'}(\xi',x_n):=\widehat{v'}(\xi',x_n) + (i\xi'/|\xi'|)\widehat{v_n}(\xi',x_n)$.
Then, there holds
\begin{align}
    (|\xi'|^2 - \partial_{x_n}^2)\widehat{w'}
    &=
    (|\xi'|^2 - \partial_{x_n}^2)\widehat{v'} + \frac{i\xi'}{|\xi'|}(|\xi'|^2 - \partial_{x_n}^2)\widehat{v_n}\\
    &=
    - i\xi'\widehat{q} - \frac{i\xi'}{|\xi'|} \partial_{x_n}\widehat{q}\\
    &=
    -\frac{i\xi'}{|\xi'|}(|\xi'|+\partial_{x_n})\widehat{q}
    =
    0.
\end{align}
From $\widehat{w'}(\cdot,x_n) \in \mathscr{S}'(\mathbb{R}^{n-1})$, it follows that
\begin{align}
    \begin{dcases}
        (|\xi'| + \partial_{x_n})\widehat{w'} = 0, &\qquad x_n>0,\\
        \widehat{w'}(\xi',0) = \widehat{b'}(\xi') + \frac{i\xi'}{|\xi'|}\widehat{b_n}(\xi'), &\qquad x_n=0.
    \end{dcases}
\end{align}
Hence, we obtain 
\begin{align}
    \widehat{w'}(\xi',x_n) 
    = 
    e^{-x_n|\xi'|}
    \left( \widehat{b'}(\xi') + \frac{i\xi'}{|\xi'|}\widehat{b_n}(\xi') \right),
\end{align}
which implies 
\begin{align}\label{v'}
    \widehat{v'}(\xi',x_n)
    =
    e^{-x_n|\xi'|}
    \left\{ b'(\xi') - x_n\frac{\xi'}{|\xi'|} \left( \xi' \cdot \widehat{b'}(\xi')\right) + \frac{i\xi'}{|\xi'|}\widehat{b_n}(\xi') \right\}.
\end{align}
Hence, we see by \eqref{v_n} and \eqref{v'} that
\begin{align}
    \widehat{u'}(\xi',x_n)
    ={}&
    e^{-x_n|\xi'|}
    \left\{ a'(\xi') - x_n\frac{\xi'}{|\xi'|} \left( \xi' \cdot \widehat{a'}(\xi')\right) + \frac{i\xi'}{|\xi'|}\widehat{a_n}(\xi') \right\}\\
    &
    -
    e^{-x_n|\xi'|}
    \left\{ \widehat{(\uw)'}(\xi',0) - x_n\frac{\xi'}{|\xi'|} \left( \xi' \cdot \widehat{(\uw)'}(\xi',0)\right) + \frac{i\xi'}{|\xi'|}\widehat{\uw_n}(\xi',0) \right\}\\
    &
    +
    \widehat{(\uw)'}(\xi',x_n),\label{hu'-1}\\
    \widehat{u_n}(\xi',x_n) 
    ={}& 
    e^{-x_n|\xi'|}
    \left\{
    ( 1 + x_n|\xi'|)\widehat{a_n}(\xi')
    -ix_n\xi'\cdot \widehat{a'}(\xi')
    \right\}\\
    &
    -
    e^{-x_n|\xi'|}
    \left\{
    ( 1 + x_n|\xi'|)\widehat{\uw_n}(\xi',0)
    -ix_n\xi'\cdot \widehat{(\uw)'}(\xi',0)
    \right\}\\
    &
    +
    \widehat{\uw_n}(\xi',x_n).\label{hun-1}
\end{align}
In order to obtain the explicit formula for $\widehat{u}(\xi',x_n)$, we should compute $\widehat{\uw}(\xi',x_n)$.
Since $\uw=(-\Delta)^{-1}\mathbb{P}\div \Fw$, where $\mathbb{P} := I + \nabla \div (-\Delta)^{-1}$ denotes the Helmholtz projection on $\mathbb{R}^n$, we have
\begin{align}
    \mathscr{F}_{\mathbb{R}^n}\left[\uw_k\right](\xi)
    ={}&
    \frac{1}{|\xi|^2}
    \sum_{\ell,m=1}^n
    \left( \delta_{k,\ell} - \frac{\xi_k\xi_{\ell}}{|\xi|^2} \right)i \xi_{m} \mathscr{F}_{\mathbb{R}^n} \left[\Fw_{\ell,m}\right](\xi), \\
    ={}&
    \frac{1}{|\xi'|^2 + \xi_n^2}
    \sum_{m=1}^n
    i \xi_{m} \mathscr{F}_{\mathbb{R}^n} \left[\Fw_{k,m}\right](\xi)\\
    &
    -
    \frac{1}{(|\xi'|^2 + \xi_n^2)^2}\sum_{\ell,m=1}^n{i\xi_k\xi_{\ell}\xi_m}  \mathscr{F}_{\mathbb{R}^n} \left[\Fw_{\ell,m}\right](\xi), 
\end{align}
for $k=1,...,n$,
where $\xi = (\xi',\xi_n)$.
For $k=1,...,n-1$, we have
\begin{align}
    \mathscr{F}_{\mathbb{R}^n}\left[\uw_k\right](\xi)
    ={}
    &
    \frac{1}{|\xi'|^2 + \xi_n^2}
    \sum_{m=1}^{n-1}
    i \xi_m \mathscr{F}_{\mathbb{R}^n} \left[\Fw_{k,m}\right](\xi)\\
    &
    +
    \frac{i\xi_n}{|\xi'|^2 + \xi_n^2}
    \mathscr{F}_{\mathbb{R}^n} \left[\Fw_{k,n}\right](\xi)\\
    &
    -
    \frac{1}{(|\xi'|^2 + \xi_n^2)^2}\sum_{\ell,m=1}^{n-1}{i\xi_k\xi_{\ell}\xi_m}  \mathscr{F}_{\mathbb{R}^n} \left[\Fw_{\ell,m}\right](\xi)\\
    &
    -
    \frac{i\xi_n}{(|\xi'|^2 + \xi_n^2)^2}\sum_{\ell=1}^{n-1}
    {\xi_k\xi_{\ell}}  
    \mathscr{F}_{\mathbb{R}^n}\left[\Fw_{\ell,n}+\Fw_{n,\ell}\right](\xi)\\
    &
    -
    \frac{\xi_n^2}{(|\xi'|^2 + \xi_n^2)^2}
    i\xi_k
    \mathscr{F}_{\mathbb{R}^n}\left[\Fw_{n,n}\right](\xi) \label{Fukw}
\end{align}
and
\begin{align}
    \mathscr{F}_{\mathbb{R}^n}\left[\uw_n\right](\xi)
    ={}
    &
    \frac{1}{|\xi'|^2 + \xi_n^2}
    \sum_{m=1}^{n-1}
    i \xi_m \mathscr{F}_{\mathbb{R}^n} \left[\Fw_{n,m}\right](\xi)\\
    &
    +
    \frac{i\xi_n}{|\xi'|^2 + \xi_n^2}
    \mathscr{F}_{\mathbb{R}^n} \left[\Fw_{n,n}\right](\xi)\\
    &
    -
    \frac{i\xi_n}{(|\xi'|^2 + \xi_n^2)^2}\sum_{\ell,m=1}^{n-1}{\xi_{\ell}\xi_m}  \mathscr{F}_{\mathbb{R}^n} \left[\Fw_{\ell,m}\right](\xi)\\
    &
    -
    \frac{\xi_n^2}{(|\xi'|^2 + \xi_n^2)^2}\sum_{\ell=1}^{n-1}
    {i\xi_{\ell}}  
    \mathscr{F}_{\mathbb{R}^n}\left[\Fw_{\ell,n}+\Fw_{n,\ell}\right](\xi)\\
    &
    -
    \frac{i\xi_n^3}{(|\xi'|^2 + \xi_n^2)^2}
    \mathscr{F}_{\mathbb{R}^n}\left[\Fw_{n,n}\right](\xi).\label{Funw}
\end{align}
Here, as it holds
\begin{align}
    &
    \frac{i\xi_n}{|\xi'|^2 + \xi_n^2}
    \mathscr{F}_{\mathbb{R}^n} \left[\Fw_{n,n}\right](\xi)
    -
    \frac{i\xi_n^3}{(|\xi'|^2 + \xi_n^2)^2}
    \mathscr{F}_{\mathbb{R}^n}\left[\Fw_{n,n}\right](\xi)
    =
    \frac{i\xi_n|\xi'|^2}{(|\xi'|^2 + \xi_n^2)^2}
    \mathscr{F}_{\mathbb{R}^n} \left[\Fw_{n,n}\right](\xi),
\end{align}
and
\begin{align}
    &
    \frac{1}{|\xi'|^2 + \xi_n^2}
    \sum_{m=1}^{n-1}
    i \xi_m \mathscr{F}_{\mathbb{R}^n} \left[\Fw_{n,m}\right](\xi)
    -
    \frac{\xi_n^2}{(|\xi'|^2 + \xi_n^2)^2}\sum_{\ell=1}^{n-1}
    {i\xi_{\ell}}  
    \mathscr{F}_{\mathbb{R}^n}\left[\Fw_{n,\ell}\right](\xi)\\
    &
    \quad
    =
    \frac{|\xi'|^2}{(|\xi'|^2 + \xi_n^2)^2}\sum_{\ell=1}^{n-1}
    {i\xi_{\ell}}  
    \mathscr{F}_{\mathbb{R}^n}\left[\Fw_{n,\ell}\right](\xi)
\end{align}
we have
\begin{align}
    \mathscr{F}_{\mathbb{R}^n}\left[\uw_n\right](\xi)
    ={}
    &
    \frac{|\xi'|^2}{(|\xi'|^2 + \xi_n^2)^2}\sum_{\ell=1}^{n-1}
    {i\xi_{\ell}}  
    \mathscr{F}_{\mathbb{R}^n}\left[\Fw_{n,\ell}\right](\xi)\\
    &
    +
    \frac{i\xi_n|\xi'|^2}{(|\xi'|^2 + \xi_n^2)^2}
    \mathscr{F}_{\mathbb{R}^n} \left[\Fw_{n,n}\right](\xi)\\
    &
    -
    \frac{i\xi_n}{(|\xi'|^2 + \xi_n^2)^2}
    \sum_{\ell,m=1}^{n-1}{\xi_{\ell}\xi_m}  \mathscr{F}_{\mathbb{R}^n} \left[\Fw_{\ell,m}\right](\xi)\\
    &
    -
    \frac{\xi_n^2}{(|\xi'|^2 + \xi_n^2)^2}\sum_{\ell=1}^{n-1}
    {i\xi_{\ell}}  
    \mathscr{F}_{\mathbb{R}^n}\left[\Fw_{\ell,n}\right](\xi).
\end{align}
Then, taking the inverse Fourier transform of \eqref{Fukw} and \eqref{Funw} with respect to $\xi_n$ and using 
\begin{align}
    \mathscr{F}_{\mathbb{R}}^{-1}
    \left[ \frac{1}{|\xi'|^2 + \xi_n^2} \right](z_n) 
    &= 
    \frac{1}{2|\xi'|}e^{-|\xi'||z_n|},\\
    \mathscr{F}_{\mathbb{R}}^{-1}
    \left[ \frac{i\xi_n}{|\xi'|^2 + \xi_n^2} \right](z_n) 
    &= 
    -\frac{1}{2}\operatorname{sgn}(z_n)e^{-|\xi'||z_n|},\\
    \mathscr{F}_{\mathbb{R}}^{-1}
    \left[ \frac{1}{(|\xi'|^2 + \xi_n^2)^2} \right](z_n) 
    &= 
    \frac{1}{4|\xi'|^3}e^{-|\xi'||z_n|}(1+|\xi'||z_n|),\\
    \mathscr{F}_{\mathbb{R}}^{-1}
    \left[ \frac{i\xi_n}{(|\xi'|^2 + \xi_n^2)^2} \right](z_n) 
    &= 
    -
    \frac{z_n}{4|\xi'|}e^{-|\xi'||z_n|},\\
    \mathscr{F}_{\mathbb{R}}^{-1}
    \left[ \frac{\xi_n^2}{(|\xi'|^2 + \xi_n^2)^2} \right](z_n) 
    &= 
    \frac{1}{4|\xi'|}e^{-|\xi'||z_n|}(1-|\xi'||z_n|),
\end{align} 
we have
\begin{align}
    \widehat{\uw_k}(\xi',x_n)
    ={}
    &
    \frac{1}{2}
    \int_{\mathbb{R}}
    K^{(1)}(\xi',x_n-y_n)
    \sum_{m=1}^{n-1}
    \frac{i \xi_m }{|\xi'|}
    \widehat{\Fw_{k,m}}(\xi',y_n)
    dy_n\\
    &
    -
    \frac{1}{2}
    \int_{\mathbb{R}}
    K^{(2)}(\xi',x_n-y_n)
    \widehat{\Fw_{k,n}}(\xi',y_n)
    dy_n\\
    &
    -
    \frac{1}{4}
    \int_{\mathbb{R}}
    K^{(3)}(\xi',x_n-y_n)
    \sum_{\ell,m=1}^{n-1}
    \frac{i\xi_k\xi_{\ell}\xi_m}{|\xi'|^3}
    \widehat{\Fw_{\ell,m}}(\xi',y_n)dy_n\\
    &
    +
    \frac{1}{4}
    \int_{\mathbb{R}}
    K^{(4)}(\xi',x_n-y_n)
    \sum_{\ell=1}^{n-1}
    \frac{\xi_k\xi_{\ell}}{|\xi'|^2}  
    \left( \widehat{\Fw_{\ell,n}}+\widehat{\Fw_{n,\ell}} \right)(\xi',y_n)
    dy_n\\
    &
    -
    \frac{1}{4}
    \int_{\mathbb{R}}
    K^{(5)}(\xi',x_n-y_n)
    \frac{i\xi_k}{|\xi'|}
    \widehat{\Fw_{n,n}}(\xi',y_n)
    dy_n
\end{align}
for $k=1,...,n-1$ and 
\begin{align}
    \widehat{\uw_n}(\xi',x_n)
    ={}&
    \frac{1}{4}
    \int_{\mathbb{R}}
    K^{(3)}(\xi',x_n-y_n)
    \sum_{\ell=1}^{n-1}
    \frac{i\xi_{\ell}}{|\xi'|}
    \widehat{\Fw_{n,\ell}}(\xi',y_n)
    dy_n\\
    &
    +
    \frac{1}{4}
    \int_{\mathbb{R}}
    K^{(4)}(\xi',x_n-y_n)
    \sum_{\ell,m=1}^{n-1}
    \frac{\xi_{\ell}\xi_m}{|\xi'|^2}
    \widehat{\Fw_{\ell,m}}(\xi',y_n)
    dy_n\\
    &
    -
    \frac{1}{4}
    \int_{\mathbb{R}}
    K^{(4)}(\xi',x_n-y_n)
    \widehat{\Fw_{n,n}}(\xi',y_n)
    dy_n\\
    &
    -
    \frac{1}{4}
    \int_{\mathbb{R}}
    K^{(5)}(\xi',x_n-y_n)
    \sum_{\ell=1}^{n-1}
    \frac{i\xi_{\ell}}{|\xi'|}  
    \widehat{\Fw_{\ell,n}}(\xi',y_n)dy_n,
\end{align}
where we have defined five functions:
\begin{align}
    K^{(1)}(\xi',z_n)
    &=
    e^{-|\xi'||z_n|},\\
    K^{(2)}(\xi',z_n)
    &=
    \operatorname{sgn}(z_n)
    e^{-|\xi'||z_n|},\\
    K^{(3)}(\xi',z_n)
    &=
    (1+|\xi'||z_n|)
    e^{-|\xi'||z_n|},\\
    K^{(4)}(\xi',z_n)
    &=
    |\xi'|z_n
    e^{-|\xi'||z_n|},\\
    K^{(5)}(\xi',z_n)
    &=
    (1-|\xi'||z_n|)
    e^{-|\xi'||z_n|}.
\end{align}
For a locally integrable function $h=h(y_n):(0,\infty) \to \mathbb{C}$, we define 
\begin{align}
    L^{(j,\pm)}[h](\xi',x_n)
    :=
    \int_{0}^{\infty}
    \left(
    K^{(j)}(\xi',x_n-y_n)
    \pm
    K^{(j)}(\xi',x_n+y_n)
    \right)
    h(y_n)
    dy_n
\end{align}
for $j=1,2,3,4,5$.
Here, we note that there holds
\begin{align}
    \int_{\mathbb{R}}
    K^{(j,\pm)}(\xi',x_n-y_n)h^{(\pm)}(y_n)dy_n
    =
    L^{(j,\pm)}[h](\xi',x_n),
\end{align}
where
\begin{align}
    h^{(\pm)}(y_n)
    :=
    \begin{cases}
        h(y_n), & (y_n>0),\\
        \pm h(y_n), & (y_n<0).
    \end{cases}
\end{align}
Thus, we have
\begin{align}
    \widehat{\uw_k}(\xi',x_n)
    ={}
    &
    \frac{1}{2}
    \sum_{m=1}^{n-1}
    \frac{i \xi_m }{|\xi'|}
    L^{(1,+)}\left[ \widehat{F_{k,m}}(\xi',\cdot) \right](x_n)
    -
    \frac{1}{2}
    L^{(2,-)}\left[ \widehat{F_{k,n}}(\xi',\cdot) \right](x_n)\\
    &
    -
    \frac{1}{4}
    \sum_{\ell,m=1}^{n-1}
    \frac{i\xi_k\xi_{\ell}\xi_m}{|\xi'|^3}
    L^{(3,+)}\left[ \widehat{F_{\ell,m}}(\xi',\cdot) \right](x_n)\\
    &
    +
    \frac{1}{4}
    \sum_{\ell=1}^{n-1}
    \frac{\xi_k\xi_{\ell}}{|\xi'|^2}  
    \left( 
    L^{(4,+)}\left[ \widehat{F_{n,\ell}}(\xi',\cdot) \right](x_n)
    +
    L^{(4,-)}\left[ \widehat{F_{\ell,n}}(\xi',\cdot) \right](x_n)
    \right)\\
    &
    -
    \frac{1}{4}
    \frac{i\xi_k}{|\xi'|}
    L^{(5,-)}\left[ \widehat{F_{n,n}}(\xi',\cdot) \right](x_n)\label{hu'-2}
\end{align}
for $k=1,...,n-1$ and 
\begin{align}
    \widehat{\uw_n}(\xi',x_n)
    ={}&
    \frac{1}{4}
    \sum_{\ell=1}^{n-1}
    \frac{i\xi_{\ell}}{|\xi'|}
    L^{(3,+)}\left[ \widehat{F_{n,\ell}}(\xi',\cdot) \right](x_n)
    \\
    &
    +
    \frac{1}{4}
    \sum_{\ell,m=1}^{n-1}
    \frac{\xi_{\ell}\xi_m}{|\xi'|^2}
    L^{(4,+)}\left[ \widehat{F_{\ell,m}}(\xi',\cdot) \right](x_n)
    \\
    &
    -
    \frac{1}{4}
    L^{(4,-)}\left[ \widehat{F_{n,n}}(\xi',\cdot) \right](x_n)\\
    &
    -
    \frac{1}{4}
    \sum_{\ell=1}^{n-1}
    \frac{i\xi_{\ell}}{|\xi'|}  
    L^{(5,-)}\left[ \widehat{F_{\ell,n}}(\xi',\cdot) \right](x_n).\label{hun-2}
\end{align}
Hence, combining \eqref{hu'-1}, \eqref{hun-1}, \eqref{hu'-2}, and \eqref{hun-2}, 
we obtain the following theorem.
\begin{thm}
Let $n \geq 2$ be an integer.
Then, the solution $u=\left(u'(x),u_{n}(x)\right)=\left(u_{1}(x),...,u_{n-1}(x),u_{n}(x)\right)$ of \eqref{eq:NS} is given by
\begin{align}
    u
    ={}&
    \mathcal{U}^{\rm boundary}[a]
    +
    \mathcal{U}^{\rm force}[F].
\end{align}
Here, we have defined 
\begin{align}
    (\mathcal{U}^{\rm boundary})'[a](x)
    ={}&
    e^{-x_n|\nabla'|}
    \left[ a' + \mathcal{R}' \left(x_n \nabla' \cdot {a'}\right) + \mathcal{R}'{a_n} \right](x') 
    \label{U'-2}\\
    \mathcal{U}_n^{\rm boundary}[a](x) 
    ={}& 
    e^{-x_n|\nabla'|}
    \left[
    ( 1 + x_n|\nabla'|){a_n}
    -x_n\nabla'\cdot {a'}
    \right](x') 
\end{align}
and
\begin{align}
    \mathcal{U}^{\rm force}[F]:= - \mathcal{U}^{\rm boundary}[\uw[F](\cdot,0)] + \uw[F],
\end{align}
where $\uw = \uw[F]$ is the solution to \eqref{eq:lin-2} on the whole space, which is explicitly given by  
\begin{align}
    {\uw_k}[F]
    ={}
    &
    \frac{1}{2}
    \sum_{m=1}^{n-1}
    \mathcal{R}_m
    \mathcal{L}^{(1,+)}\left[ {F_{k,m}} \right]
    -
    \frac{1}{2}
    \mathcal{L}^{(2,-)}\left[ {F_{k,n}} \right]
    +
    \frac{1}{4}
    \sum_{\ell,m=1}^{n-1}
    \mathcal{R}_k\mathcal{R}_{\ell}\mathcal{R}_m
    \mathcal{L}^{(3,+)}\left[ {F_{\ell,m}} \right]\\
    &
    -
    \frac{1}{4}
    \sum_{\ell=1}^{n-1}
    \mathcal{R}_k\mathcal{R}_{\ell}  
    \left( 
    \mathcal{L}^{(4,+)}\left[ {F_{n,\ell}} \right]
    +
    \mathcal{L}^{(4,-)}\left[ {F_{\ell,n}} \right]
    \right)
    -
    \frac{1}{4}
    \mathcal{R}_k
    \mathcal{L}^{(5,-)}\left[ {F_{n,n}} \right]
\end{align}
for $k=1,...,n-1$ and 
\begin{align}
    {\uw_n}[F]
    ={}&
    \frac{1}{4}
    \sum_{\ell=1}^{n-1}
    \mathcal{R}_{\ell}
    \mathcal{L}^{(3,+)}\left[ {F_{n,\ell}} \right]
    -
    \frac{1}{4}
    \sum_{\ell,m=1}^{n-1}
    \mathcal{R}_{\ell}\mathcal{R}_m
    \mathcal{L}^{(4,+)}\left[ {F_{\ell,m}} \right]
    \\
    &
    -
    \frac{1}{4}
    \mathcal{L}^{(4,-)}\left[ {F_{n,n}} \right]
    -
    \frac{1}{4}
    \sum_{\ell=1}^{n-1}
    \mathcal{R}_{\ell}
    \mathcal{L}^{(5,-)}\left[ {F_{\ell,n}} \right].
\end{align}
Here, $\mathcal{R}'=(\mathcal{R}_1,...,\mathcal{R}_{n-1})$; $\mathcal{R}_{\ell}:=\partial_{x_\ell}/|\nabla'|$ $(\ell=1,...,n-1)$ are the Riesz transforms on $\mathbb{R}^{n-1}$, and $\mathcal{L}^{(j,\pm)}$ $(j=1,2,3,4,5)$ are defined as
\begin{align}
    \mathcal{L}^{(j,\pm)}[f](x)
    :={}&
    \mathscr{F}^{-1}_{\mathbb{R}^{n-1}_{\xi'}}
    \left[
    L^{(j,\pm)}
    \left[\widehat{f}(\xi',\cdot)\right](x_n)
    \right](x')\\
    ={}&
    \int_0^{\infty}
    \left( K^{(j)}(D',x_n-y_n) \pm K^{(j)}(D',x_n+y_n) \right)
    f(x',y_n)
    dy_n,
\end{align}
where, $K^{(j)}(D',x_n):=\mathscr{F}^{-1}_{\mathbb{R}^{n-1}}K^{(j)}(\xi',x_n)\mathscr{F}_{\mathbb{R}^{n-1}}$ are the Fourier multipliers on $\mathbb{R}^{n-1}$.
\end{thm}
\begin{rem}
    Since the divergence-free condition yields $(u \cdot \nabla)u = \div(u \otimes u)$, we see that the first equation of \eqref{eq:NS} is equivalent
    \begin{align}
        -\Delta u + \nabla p = \div (F - u \otimes u ).
    \end{align}
    Thus, we may rewrite \eqref{eq:NS} as
    \begin{align}
        u = \mathcal{U}^{\rm boundary}[a] + \mathcal{U}^{\rm force}[F - u \otimes u].
    \end{align}
\end{rem}

Next, we consider the estimate of the linear solution to \eqref{eq:lin-1}.
First of all, we recall the following estimate.
\begin{lemm}[\cite{Iwa-15}]\label{lemm:dissp}
    Let $n \geq 2$ be an integer.
    Then, there exist positive constants $c=c(n)$ and $C=C(n)$ such that 
    for any $1 \leq p \leq \infty$ and $j \in \mathbb{Z}$, it holds
    \begin{align}
        \n{\Delta_j e^{-|\nabla'|x_n}f}_{L^p_{x'}} \leq C e^{-c 2^j x_n } \n{ \Delta_j f}_{L^p_{x'}}  
    \end{align}
    for all $x_n>0$ and $f \in \mathscr{S}'(\mathbb{R}^{n-1})$ with $\Delta_j f \in L^p(\mathbb{R}^{n-1})$.
\end{lemm}
Making use of Lemma \ref{lemm:dissp} and the Bernstein inequality, we obtain the following lemma.
\begin{lemm}\label{lemm:max-reg}
    Let $n \geq 2$ be an integer.
    There exists a positive constant $C=C(n)$ such that for $1 \leq p,r \leq \infty$, and $1 \leq q \leq q_1 \leq \infty$, the solution $u$ to \eqref{eq:lin-1} with the boundary data $a \in \dB_{p,r}^{\frac{n-1}{p}-1}(\mathbb{R}^{n-1})$ and the external force $F \in \widetilde{L^q_{x_n}}(\dB_{p,r}^{\frac{n-1}{p}+\frac{1}{q}-2})_{x'}(\mathbb{R}^n_+)$ satisfies
    \begin{align}
        \n{u}_{\widetilde{L^{q_1}_{x_n}}(\dB_{p,r}^{\frac{n-1}{p}+\frac{1}{q_1}-1})_{x'}}
        \leq
        C
        \left(
        \n{a}_{\dB_{p,r}^{\frac{n-1}{p}-1}}
        +
        \n{F}_{\widetilde{L^q_{x_n}}(\dB_{p,r}^{\frac{n-1}{p}+\frac{1}{q}-2})_{x'}}
        \right).
    \end{align}
\end{lemm}
\begin{proof}
By the Bernstein inequality and Lemma \ref{lemm:dissp}, we have
\begin{align}
    \n{\Delta_j u(\cdot,x_n)}_{L^p_{x'}}
    \leq{}&
    Ce^{-c2^jx_n}\left( 1 + 2^j x_n \right)\n{\Delta_j a}_{L^p_{x'}}\\
    &+
    Ce^{-c2^jx_n}\left( 1 + 2^j x_n \right)\n{\Delta_j \uw(\cdot,0)}_{L^p_{x'}}
    +
    C\n{\Delta_j \uw(\cdot,x_n)}_{L^p_{x'}}\\
    \leq{}&
    Ce^{-c2^jx_n}\n{\Delta_j a}_{L^p_{x'}}
    +
    Ce^{-c2^jx_n}\n{\Delta_j \uw(\cdot,0)}_{L^p_{x'}}\\
    &+
    C\sum_{k=1}^5\n{\Delta_j \mathcal{L}^{(k,\pm)}[F](\cdot,x_n)}_{L^p_{x'}}.
\end{align}
For the estimate of $\uw(\cdot,0)$, it holds
\begin{align}
    &{\uw_k}(x',0)
    ={}
    \frac{1}{2}
    \sum_{m=1}^{n-1}
    \mathcal{R}_m
    \mathcal{L}^{(1,+)}\left[ {F_{k,m}} \right](x',0)
    +
    \frac{1}{2}
    \mathcal{L}^{(1,+)}\left[ {F_{k,n}} \right](x',0)\\
    &\quad
    +
    \frac{1}{4}
    \sum_{\ell,m=1}^{n-1}
    \mathcal{R}_k\mathcal{R}_{\ell}\mathcal{R}_m
    \mathcal{L}^{(3,+)}\left[ {F_{\ell,m}} \right](x',0)
    -
    \frac{1}{4}
    \sum_{\ell=1}^{n-1}
    \mathcal{R}_k\mathcal{R}_{\ell}  
    \mathcal{L}^{(4,-)}\left[ {F_{\ell,n}} \right](x',0)
\end{align}
for $k=1,...,n-1$ and 
\begin{align}
    {\uw_n}(x',0)
    ={}&
    \frac{1}{4}
    \sum_{\ell=1}^{n-1}
    \mathcal{R}_{\ell}
    \mathcal{L}^{(3,+)}\left[ {F_{n,\ell}} \right](x',0)
    -
    \frac{1}{4}
    \mathcal{L}^{(4,-)}\left[ {F_{n,n}} \right](x',0),
\end{align}
where we have used 
\begin{align}
    &\mathcal{L}^{(2,-)}\left[ f \right](x',0)
    =
    -
    \mathcal{L}^{(1,+)}\left[ f \right](x',0),\\
    &
    \mathcal{L}^{(4,+)}\left[ f \right](x',0)
    =
    \mathcal{L}^{(5,-)}\left[ f \right](x',0)
    =0.
\end{align}
As there holds
\begin{align}
    &\mathcal{L}^{(1,+)}\left[ f \right](x',0)
    =
    2
    \int_0^{\infty}
    e^{-|\nabla'|y_n}f(x',y_n)dy_n,\\
    &\mathcal{L}^{(3,+)}\left[ f \right](x',0)
    =
    2
    \int_0^{\infty}
    (1+y_n|\nabla'|)
    e^{-|\nabla'|y_n}f(x',y_n)dy_n,\\
    &\mathcal{L}^{(4,-)}\left[ f \right](x',0)
    =
    -
    2
    \int_0^{\infty}
    y_n|\nabla'|
    e^{-|\nabla'|y_n}f(x',y_n)dy_n.
\end{align}
Thus, we see that 
\begin{align}
    e^{-c2^jx_n}
    \n{\Delta_j \uw(\cdot,0)}_{L^p_{x'}}
    \leq{}&
    C
    e^{-c2^jx_n}
    \n{\Delta_j \mathcal{L}^{(1,+)}\left[ F \right](x',0)}_{L^p_{x'}}\\
    &
    +
    C
    e^{-c2^jx_n}
    \n{\Delta_j \mathcal{L}^{(3,+)}\left[ F \right](x',0)}_{L^p_{x'}}\\
    &
    +
    C
    e^{-c2^jx_n}
    \n{\Delta_j \mathcal{L}^{(4,-)}\left[ F \right](x',0)}_{L^p_{x'}}\\
    \leq{}&
    C
    e^{-c2^jx_n}
    \int_0^{\infty}e^{-c2^jy_n}(1+2^jy_n)\n{\Delta_j F(\cdot,y_n)}_{L^p_{x'}}dy_n\\
    \leq{}&
    C
    \int_0^{\infty}e^{-c2^j|x_n-y_n|}\n{\Delta_j F(\cdot,y_n)}_{L^p_{x'}}dy_n.
\end{align}
For the estimates of $\mathcal{L}^{(k,\pm)}[F]$ ($k=1,2,3,4,5$), we have
\begin{align}
    \n{\mathcal{L}^{(k,\pm)}[F](\cdot,x_n)}_{L^p_{x'}}
    \leq{}&
    C
    \int_0^{\infty}
    \left\{
    (1+2^j|x_n-y_n|)e^{-c2^j|x_n-y_n|}\right.\\
    &\qquad \left.
    +
    \left(1+2^j(x_n+y_n)\right)e^{-c2^j(x_n+y_n)}
    \right\}
    \n{\Delta_j F(\cdot,y_n)}_{L^p_{x'}}
    dy_n\\
    \leq{}&
    C
    \int_0^{\infty}e^{-c2^j|x_n-y_n|}\n{\Delta_j F(\cdot,y_n)}_{L^p_{x'}}dy_n.
\end{align}
Thus, we obtain 
\begin{align}\label{est-lin-1}
    \begin{split}
    \n{\Delta_j u(\cdot,x_n)}_{L^p_{x'}}
    \leq{}&
    Ce^{-c2^jx_n}\n{\Delta_j a}_{L^p_{x'}}\\
    &+
    C
    \int_0^{\infty}e^{-c2^j|x_n-y_n|}\n{\Delta_j F(\cdot,y_n)}_{L^p_{x'}}dy_n.
    \end{split}
\end{align}
Taking $L^{q_1}(0,\infty)$ norm with respect to $x_n$ and using the Hausdorff--Young inequality, we obtain 
\begin{align}
    \n{\Delta_j u}_{L^{q_1}(0,\infty;L^p)}
    \leq{}&
    C\n{e^{-c2^jx_n}}_{L^{q_1}_{x_n}(0,\infty)}\n{\Delta_j a}_{L^p_{x'}}
    +
    C
    \n{e^{-c2^j|\cdot|}}_{L^{q_2}(\mathbb{R})}
    \n{\Delta_j F}_{L^{q}(0,\infty;L^p)}\\
    ={}&
    C2^{-\frac{1}{q_1}j}\n{\Delta_j a}_{L^p_{x'}}
    +
    C
    2^{(-\frac{1}{q_1}+\frac{1}{q}-1)j}
    \n{\Delta_j F}_{L^{q}(0,\infty;L^p)},
\end{align}
where
$1/q_2:=1/q_1-1/q+1$.
Multiplying this by $2^{(\frac{n-1}{p}+\frac{1}{q_1}-1)j}$ and taking $\ell^r(\mathbb{Z})$ norm, 
we complete the proof.
\end{proof}

\section{Proofs of main theorems}\label{sec:pf}
In this section, we prove our main theorems.
Before starting the proof, we introduce a lemma for bilinear estimate.
\begin{lemm}\label{lemm:bilin}
    Let $n\geq3$ be an integer.
    Let $1 \leq q,r\leq \infty$ and let
    \begin{align}
        1 \leq p <q_*'(n-1), 
        \qquad 
        q_*:=\max\{2,q\}, 
    \end{align}
    where $q_*':=q_*/(q_*-1)$ denotes the H\"older conjugate of $q_*$.
    We additionally assume that $q < \infty$ if $n=3$.
    Then, there exists a positive constant $C=C(p,q,r)$ such that 
    \begin{align}
        \n{fg}_{\widetilde{L^{\frac{q_*}{2}}_{x_n}}(\dB_{p,r}^{\frac{n-1}{p}+\frac{2}{q_*}-2})_{x'}}
        \leq
        C
        \n{f}_{\widetilde{L^{q_*}_{x_n}}(\dB_{p,r}^{\frac{n-1}{p}+\frac{1}{q_*}-1})_{x'}}
        \n{g}_{\widetilde{L^{q_*}_{x_n}}(\dB_{p,r}^{\frac{n-1}{p}+\frac{1}{q_*}-1})_{x'}}
    \end{align}
    for all $f,g \in \widetilde{L^{q_*}_{x_n}}(\dB_{p,r}^{\frac{n-1}{p}+\frac{1}{q_*}-1})_{x'}$.
\end{lemm}
\begin{proof}
Although the proof is based on the standard para product method (see \cite{Bah-Che-Dan-11} for the detail), 
we provide the precise proof as there are some steps that need a little bit complicated arguments.
Let us decompose the product $fg$ as 
\begin{align}
    fg = T_fg + R(f,g) + T_gf,
\end{align}
where
\begin{align}
    T_fg := \sum_{k \in \mathbb{Z}} \left(\sum_{\ell \leq k -3}\Delta_{\ell}f\right)\Delta_kg,\quad
    R(f,g) := \sum_{|k-\ell| \leq 2} \Delta_k f \Delta_{\ell}g.
\end{align}
We note that 
\begin{align}
    &\Delta_j T_fg = \sum_{|j - k | \leq 2 }\Delta_j\left\{ \left(\sum_{\ell \leq k -3}\Delta_{\ell}f\right) \Delta_kg \right\},\\
    &\Delta_j R(f,g) = \sum_{k \geq j-4}\sum_{|k-\ell| \leq 2} \Delta_j \left( \Delta_k f \Delta_{\ell} g \right)
\end{align}
holds for all $j \in \mathbb{Z}$.

By the H\"older inequality, we have
\begin{align}
    &\n{\Delta_jT_fg}_{{L^{\frac{q_*}{2}}_{x_n}}(0,\infty;L^p_{x'})}\\
    &\quad\leq{}
    C
    \sum_{|j - k | \leq 2 }
    \left(\sum_{\ell \leq k -3}\n{\Delta_{\ell}f}_{L^{q_*}_{x_n}(0,\infty;L^{\infty}_{x'})}\right) \n{\Delta_kg}_{L^{q_*}_{x_n}(0,\infty;L^p_{x'})}\\
    &\quad\leq{}
    C
    \sum_{|j - k | \leq 2 }
    \left(\sum_{\ell \leq k -3}2^{\frac{n-1}{p}\ell}\n{\Delta_{\ell}f}_{L^{q_*}_{x_n}(0,\infty;L^p_{x'})}\right) \n{\Delta_kg}_{L^{q_*}_{x_n}(0,\infty;L^p_{x'})}\\
    &\quad\leq{}
    C
    \n{f}_{\widetilde{L^{q_*}_{x_n}}(\dB_{p,r}^{\frac{n-1}{p}+\frac{1}{q_*}-1})_{x'}}
    \sum_{|j - k | \leq 2 }
    2^{(1-\frac{1}{q_*})k}\n{\Delta_kg}_{L^{q_*}_{x_n}(0,\infty;L^p_{x'})}.
\end{align}
Multiplying this by $2^{(\frac{n-1}{p}+\frac{2}{q_*}-2)j}$ and taking $\ell^r$-norm,
we have
\begin{align}
        \n{T_fg}_{\widetilde{L^{\frac{q_*}{2}}_{x_n}}(\dB_{p,r}^{\frac{n-1}{p}+\frac{2}{q_*}-2})_{x'}}
        \leq
        C
        \n{f}_{\widetilde{L^{q_*}_{x_n}}(\dB_{p,r}^{\frac{n-1}{p}+\frac{1}{q_*}-1})_{x'}}
        \n{g}_{\widetilde{L^{q_*}_{x_n}}(\dB_{p,r}^{\frac{n-1}{p}+\frac{1}{q_*}-1})_{x'}}
\end{align}
Similarly, we have
\begin{align}
        \n{T_gf}_{\widetilde{L^{\frac{q_*}{2}}_{x_n}}(\dB_{p,r}^{\frac{n-1}{p}+\frac{2}{q_*}-2})_{x'}}
        \leq
        C
        \n{f}_{\widetilde{L^{q_*}_{x_n}}(\dB_{p,r}^{\frac{n-1}{p}+\frac{1}{q_*}-1})_{x'}}
        \n{g}_{\widetilde{L^{q_*}_{x_n}}(\dB_{p,r}^{\frac{n-1}{p}+\frac{1}{q_*}-1})_{x'}}.
\end{align}

Next, we consider the estimate for $R(f,g)$ by the steps divided into three parts.
We first consider the case 
\footnote{If $q=\infty$ with $n=3$, then we see that $[2, q_*'(n-1))=[2,2)=\emptyset$.}
\begin{align}\label{case1}
    \begin{split}
    &2\leq p < q_*'(n-1),\quad 2\leq q \leq \infty \qquad {\rm with \ }n \geq 4,\\
    &2\leq p < q_*'(n-1),\quad 2\leq q <    \infty \qquad {\rm with \ }n \geq 3.
    \end{split}
\end{align}
We note that $q_*=q$ in this case.
We see that
\begin{align}
    &
    2^{(\frac{2(n-1)}{p}+\frac{2}{q}-2)j}
    \n{\Delta_jR(f,g)}_{L^{\frac{q}{2}}_{x_n}(0,\infty;L^{\frac{p}{2}})}\\
    &
    \quad
    \leq
    C
    \sum_{k \geq j-4}
    2^{(\frac{2(n-1)}{p}+\frac{2}{q}-2)(j-k)}\\
    &\qquad\qquad\qquad\times
    2^{(\frac{2(n-1)}{p}+\frac{2}{q}-2)k}
    \n{\Delta_kf}_{L^q_{x_n}(0,\infty;L^p_{x'})}
    \sum_{|\ell - k| \leq 2}
    \n{\Delta_{\ell}g}_{L^q_{x_n}(0,\infty;L^p_{x'})}\\
    &\quad 
    =C(a*b)_j,
\end{align}
where $*$ stands for the convolution on $\mathbb{Z}$, and two sequences $a=\{a_j\}_{j\in \mathbb{Z}}$ and $b=\{b_j\}_{j\in \mathbb{Z}}$ are defined by
\begin{align}
    a_j:={}&
    \begin{cases}
        2^{(\frac{2(n-1)}{p}+\frac{2}{q}-2)j} & (j \leq 4),\\
        0 & (j \geq 5),
    \end{cases}\\
    b_j:={}&
    2^{(\frac{2(n-1)}{p}+\frac{2}{q}-2)j}
    \n{\Delta_jf}_{L^q_{x_n}(0,\infty;L^p_{x'})}
    \sum_{|\ell - j| \leq 2}
    \n{\Delta_{\ell}g}_{L^q_{x_n}(0,\infty;L^p_{x'})}.
\end{align}
Using the fact that \eqref{case1} ensures $2(n-1)/p+2/q-2>0$ and the Hausdorff--Young inequality, we have
\begin{align}
    &
    \n{R(f,g)}_{\widetilde{L^{\frac{q}{2}}_{x_n}}(\dB_{p,r}^{\frac{n-1}{p}+\frac{2}{q_*}-2})_{x'}}\\
    &\quad
    \leq
    C
    \n{R(f,g)}_{\widetilde{L^{\frac{q}{2}}_{x_n}}(\dB_{\frac{p}{2},r}^{\frac{2(n-1)}{p}+\frac{2}{q_*}-2})_{x'}}\\
    &\quad 
    \leq
    C
    \n{\left\{2^{(\frac{2(n-1)}{p}+\frac{2}{q}-2)k}
    \n{\Delta_kf}_{L^q_{x_n}(0,\infty;L^p_{x'})}
    \sum_{|\ell - k| \leq 2}
    \n{\Delta_{\ell}g}_{L^q_{x_n}(0,\infty;L^p_{x'})}\right\}_{k\in \mathbb{Z}}}_{\ell^r(\mathbb{Z})}\\
    &\quad
    \leq
    C
    \n{f}_{\widetilde{L^{q}_{x_n}}(\dB_{p,r}^{\frac{n-1}{p}+\frac{1}{q}-1})_{x'}}
    \n{g}_{\widetilde{L^{q}_{x_n}}(\dB_{p,r}^{\frac{n-1}{p}+\frac{1}{q}-1})_{x'}}.
\end{align}
Next, we consider the case of 
\begin{align}\label{case2}
    2 \leq p < q_*'(n-1),\qquad 1 \leq q \leq 2.
\end{align}
We note that $q_*=q_*'=2$ holds in this case.
We see that
\begin{align}
    &
    2^{(\frac{2(n-1)}{p}-1)j}
    \n{\Delta_j R(f,g)}_{L^{1}_{x_n}(0,\infty;L^{\frac{p}{2}})}\\
    &
    \quad
    \leq
    C
    \sum_{k \geq j-4}
    2^{(\frac{2(n-1)}{p}-1)(j-k)}
    2^{(\frac{2(n-1)}{p}-1)k}
    \n{\Delta_kf}_{L^2_{x_n}(0,\infty;L^p_{x'})}
    \sum_{|\ell - k| \leq 2}
    \n{\Delta_{\ell}g}_{L^2_{x_n}(0,\infty;L^p_{x'})}.
\end{align}
Similarly as above, using the fact that \eqref{case2} ensures $2(n-1)/p-1>0$ and the Hausdorff--Young inequality, we have
\begin{align}
    &
    \n{R(f,g)}_{\widetilde{L^{1}_{x_n}}(\dB_{p,r}^{\frac{n-1}{p}-1})_{x'}}\\
    &\quad
    \leq
    C
    \n{R(f,g)}_{\widetilde{L^{1}_{x_n}}(\dB_{\frac{p}{2},r}^{\frac{2(n-1)}{p}-1})_{x'}}\\
    &\quad 
    \leq
    C
    \n{\left\{2^{(\frac{2(n-1)}{p}-1)k}
    \n{\Delta_kf}_{L^2_{x_n}(0,\infty;L^{p}_{x'})}
    \sum_{|\ell - k| \leq 2}
    \n{\Delta_{\ell}g}_{L^2_{x_n}(0,\infty;L^p_{x'})}\right\}_{k\in \mathbb{Z}}}_{\ell^r(\mathbb{Z})}\\
    &\quad
    \leq
    C
    \n{f}_{\widetilde{L^{2}_{x_n}}(\dB_{p,r}^{\frac{n-1}{p}-\frac{1}{2}})_{x'}}
    \n{g}_{\widetilde{L^{2}_{x_n}}(\dB_{p,r}^{\frac{n-1}{p}-\frac{1}{2}})_{x'}}.
\end{align}
Finally, we consider the case 
\begin{align}\label{case3}
    \begin{split}
    (p,q) &\in [1,2)\times[1,\infty] \qquad {\rm with\ }n \geq 4,\\
    (p,q) &\in [1,2)\times[1,\infty) \qquad {\rm with\ }n = 3.
    \end{split}
\end{align}
We have by the H\"older and Bernstein inequalities that 
\begin{align}
    &
    2^{(n-3+\frac{2}{q_*})j}
    \n{\Delta_j R(f,g)}_{L^{\frac{q_*}{2}}_{x_n}(0,\infty;L^1)}\\
    &
    \quad
    \leq
    C
    \sum_{k \geq j-4}
    2^{(n-3+\frac{2}{q_*})j}
    \n{\Delta_kf}_{L^{q_*}_{x_n}(0,\infty;L^p_{x'})}
    \sum_{|\ell - k| \leq 2}
    \n{\Delta_\ell g}_{L^{q_*}_{x_n}(0,\infty;L^{p'}_{x'})}\\
    &
    \quad
    \leq
    C
    \sum_{k \geq j-4}
    2^{(n-3+\frac{2}{q_*})j}
    \n{\Delta_kf}_{L^{q_*}_{x_n}(0,\infty;L^p_{x'})}
    \sum_{|\ell - k| \leq 2}
    2^{(n-1)(\frac{2}{p}-1)\ell}
    \n{\Delta_\ell g}_{L^{q_*}_{x_n}(0,\infty;L^{p}_{x'})}\\
    &
    \begin{aligned}
    \quad
    =
    C
    \sum_{k \geq j-4}
    2^{(n-3+\frac{2}{q_*})(j-k)}
    &
    2^{(\frac{n-1}{p}+\frac{1}{q_*}-1)k}
    \n{\Delta_kf}_{L^{q_*}_{x_n}(0,\infty;L^p_{x'})}\\
    &
    \times
    \sum_{|\ell - k| \leq 2}
    2^{(\frac{n-1}{p}+\frac{1}{q_*}-1)\ell}
    \n{\Delta_\ell g}_{L^{q_*}_{x_n}(0,\infty;L^{p}_{x'})}.
    \end{aligned}
\end{align}
Here $p'=p/(p-1)$ denotes the H\"older conjugate of $p$.
Using the fact that \eqref{case3} ensures $n-3+2/q_*>0$
\footnote{We remark that this fails if $n=3$ and $q=\infty$. Thus, we suppose $q < \infty$ for the case $n=3$ in \eqref{case3}.}
and the Hausdorff--Young inequality, we have
\begin{align}
    &
    \n{R(f,g)}_{\widetilde{L^{\frac{q_*}{2}}_{x_n}}(\dB_{p,r}^{\frac{n-1}{p}+\frac{2}{q_*}-2})_{x'}}\\
    &\quad
    \leq
    C
    \n{R(f,g)}_{\widetilde{L^{\frac{q_*}{2}}_{x_n}}(\dB_{1,r}^{n-3+\frac{2}{q_*}})_{x'}}\\
    &\quad 
    \leq
    C
    \left\|\left\{
    2^{(\frac{n-1}{p}+\frac{1}{q_*}-1)k}
    \n{\Delta_kf}_{L^{q_*}_{x_n}(0,\infty;L^p_{x'})}\right.\right.\\
    &\qquad \qquad \qquad
    \left.\left.
    \times
    \sum_{|\ell - k| \leq 2}
    2^{(\frac{n-1}{p}+\frac{1}{q_*}-1)\ell}
    \n{\Delta_\ell g}_{L^{q_*}_{x_n}(0,\infty;L^{p}_{x'})}
    \right\}_{k\in \mathbb{Z}}\right\|_{\ell^r(\mathbb{Z})}\\
    &\quad
    \leq
    C
    \n{f}_{\widetilde{L^{q_*}_{x_n}}(\dB_{p,r}^{\frac{n-1}{p}+\frac{1}{q_*}-1})_{x'}}
    \n{g}_{\widetilde{L^{q_*}_{x_n}}(\dB_{p,r}^{\frac{n-1}{p}+\frac{1}{q_*}-1})_{x'}}.
\end{align}
Hence, collecting above estimates, we complete the proof.
\end{proof}

Now, we are in a position to present the proofs of main results.
\begin{proof}[Proof of Theorem \ref{thm:1}]
Let $a \in \dB_{p,r}^{\frac{n-1}{p}-1}(\mathbb{R}^{n-1})$ and $F \in \widetilde{L^q_{x_n}}(\dB_{p,r}^{\frac{n-1}{p}+\frac{1}{q}-2})_{x'}(\mathbb{R}^n_+)$ satisfy \eqref{small}, where $\delta_0$ is a positive constant to be determined later. 
Let
\begin{align}
    X_{p,q,r}
    &:=
    \left\{
    \begin{aligned}
        u 
        \in \widetilde{L^{q_*}_{x_n}}(\dB_{p,r}^{\frac{n-1}{p}+\frac{1}{q_*}-1})_{x'}(\mathbb{R}^n_+)
    \end{aligned}
    \ ;\
    \| u \|_{\widetilde{L^{q_*}_{x_n}}(\dB_{p,r}^{\frac{n-1}{p}+\frac{1}{q_*}-1})_{x'}}
    \leq \varepsilon_0 
    \right\},
\end{align}
where $\varepsilon_0$ is a positive constant to be determined later.
The aim of this proof is to construct a unique function $u \in X_{p,q,r}$ satisfying 
\begin{align}
    u=\mathcal{U}^{\rm boundary}[a] + \mathcal{U}^{\rm force}[F - u \otimes u].
\end{align}
Let us define a map $\mathcal{S}[v]:=\mathcal{U}^{\rm boundary}[a] + \mathcal{U}^{\rm force}[F - v \otimes v]$ for $v \in X_{p,q,r}$.
Here, it follows from Lemmas \ref{lemm:max-reg} and \ref{lemm:bilin} that
there exists a constant $C_0=C_0(n,p,q,r) \geq 1$ such that 
\begin{align}
    \n{\mathcal{U}^{\rm boundary}[a]}_{\widetilde{L^{q_*}_{x_n}}(\dB_{p,r}^{\frac{n-1}{p}+\frac{1}{q_*}-1})_{x'}}
    &\leq
    C_0
    \n{a}_{\dB_{p,r}^{\frac{n-1}{p}-1}},\\
    \n{\mathcal{U}^{\rm force}[F]}_{\widetilde{L^{q_*}_{x_n}}(\dB_{p,r}^{\frac{n-1}{p}+\frac{1}{q_*}-1})_{x'}}
    &\leq
    C_0
    \n{F}_{\widetilde{L^{q}_{x_n}}(\dB_{p,r}^{\frac{n-1}{p}+\frac{1}{q}-2})_{x'}},
\end{align}
and
\begin{align}
    \n{\mathcal{U}^{\rm force}[v \otimes w]}_{\widetilde{L^{q_*}_{x_n}}(\dB_{p,r}^{\frac{n-1}{p}+\frac{1}{q_*}-1})_{x'}}
    &\leq
    C
    \n{v\otimes w}_{\widetilde{L^{\frac{q_*}{2}}_{x_n}}(\dB_{p,r}^{\frac{n-1}{p}+\frac{2}{q_*}-2})_{x'}}\\
    &\leq
    C_0
    \n{v}_{\widetilde{L^{q_*}_{x_n}}(\dB_{p,r}^{\frac{n-1}{p}+\frac{1}{q_*}-1})_{x'}}
    \n{w}_{\widetilde{L^{q_*}_{x_n}}(\dB_{p,r}^{\frac{n-1}{p}+\frac{1}{q_*}-1})_{x'}}
\end{align}
for all $v,w \in \widetilde{L^{q_*}_{x_n}}(\dB_{p,r}^{\frac{n-1}{p}+\frac{1}{q_*}-1})_{x'}(\mathbb{R}^{n}_+)$.
Then, we see that 
\begin{align}
    \n{\mathcal{S}[v]}_{\widetilde{L^{q_*}_{x_n}}(\dB_{p,r}^{\frac{n-1}{p}+\frac{1}{q_*}-1})_{x'}}
    &
    \leq
    C_0
    \n{a}_{\dB_{p,r}^{\frac{n-1}{p}-1}}
    +
    C_0
    \n{F}_{\widetilde{L^{q}_{x_n}}(\dB_{p,r}^{\frac{n-1}{p}+\frac{1}{q}-2})_{x'}}\\
    &\quad
    +
    C_0
    \left(
    \n{v}_{\widetilde{L^{q_*}_{x_n}}(\dB_{p,r}^{\frac{n-1}{p}+\frac{1}{q_*}-1})_{x'}}
    \right)^2\\
    &\leq
    C_0\delta_0
    +
    C_0\varepsilon_0^2
\end{align}
for all $v \in X_{p,q,r}$.
Similarly, we see by 
\begin{align}
    \mathcal{S}[v] - \mathcal{S}[w] 
    =
    -
    \mathcal{U}^{\rm force}[v\otimes(v-w) ]
    -
    \mathcal{U}^{\rm force}[(v-w)\otimes w]
\end{align}
that 
\begin{align}
    &\n{\mathcal{S}[v] - \mathcal{S}[w]}_{\widetilde{L^{q_*}_{x_n}}(\dB_{p,r}^{\frac{n-1}{p}+\frac{1}{q_*}-1})_{x'}}\\
    &\quad 
    \leq
    C_0
    \left(
    \n{v}_{\widetilde{L^{q_*}_{x_n}}(\dB_{p,r}^{\frac{n-1}{p}+\frac{1}{q_*}-1})_{x'}}
    +
    \n{w}_{\widetilde{L^{q_*}_{x_n}}(\dB_{p,r}^{\frac{n-1}{p}+\frac{1}{q_*}-1})_{x'}}
    \right)
    \n{v-w}_{\widetilde{L^{q_*}_{x_n}}(\dB_{p,r}^{\frac{n-1}{p}+\frac{1}{q_*}-1})_{x'}}\\
    &\quad
    \leq
    2C_0\varepsilon_0
    \n{v-w}_{\widetilde{L^{q_*}_{x_n}}(\dB_{p,r}^{\frac{n-1}{p}+\frac{1}{q_*}-1})_{x'}}
\end{align}
for all $v,w \in X_{p,q,r}$.
Here, let $\delta_0:= 1/(12C_0^2)$ and $\varepsilon_0=1/(4C_0)$.
Then, there holds
\begin{align}
    \n{\mathcal{S}[v]}_{\widetilde{L^{q_*}_{x_n}}(\dB_{p,r}^{\frac{n-1}{p}+\frac{1}{q_*}-1})_{x'}}
    &\leq
    \varepsilon_0,
    \\
    \n{\mathcal{S}[v] - \mathcal{S}[w]}_{\widetilde{L^{q_*}_{x_n}}(\dB_{p,r}^{\frac{n-1}{p}+\frac{1}{q_*}-1})_{x'}}
    &\leq
    \frac{1}{2}
    \n{v - w}_{\widetilde{L^{q_*}_{x_n}}(\dB_{p,r}^{\frac{n-1}{p}+\frac{1}{q_*}-1})_{x'}}
\end{align}
for all $v,w \in X_{p,q,r}$.
Hence, the contraction mapping principle implies that there exists a unique $u \in X_{p,q,r}$ such that $u = \mathcal{S}[u]$.
It follows from Lemma \ref{lemm:max-reg} that 
\begin{align}
    \n{u}_{\widetilde{L^{\infty}_{x_n}}(\dB_{p,r}^{\frac{n-1}{p}-1})_{x'}}
    &=
    \n{\mathcal{S}[u]}_{\widetilde{L^{\infty}_{x_n}}(\dB_{p,r}^{\frac{n-1}{p}-1})_{x'}}\\
    &\leq
    C
    \n{a}_{\dB_{p,r}^{\frac{n-1}{p}-1}}
    +
    C
    \n{F}_{\widetilde{L^{q}_{x_n}}(\dB_{p,r}^{\frac{n-1}{p}+\frac{1}{q}-2})_{x'}}\\
    &\quad
    +
    C
    \n{u\otimes u}_{\widetilde{L^{\frac{q_*}{2}}_{x_n}}(\dB_{p,r}^{\frac{n-1}{p}+\frac{2}{q_*}-2})_{x'}}\\
    &\leq
    C
    \n{a}_{\dB_{p,r}^{\frac{n-1}{p}-1}}
    +
    C
    \n{F}_{\widetilde{L^{q}_{x_n}}(\dB_{p,r}^{\frac{n-1}{p}+\frac{1}{q}-2})_{x'}}\\
    &\quad
    +
    C
    \left(
    \n{u}_{\widetilde{L^{q_*}_{x_n}}(\dB_{p,r}^{\frac{n-1}{p}+\frac{1}{q_*}-1})_{x'}}
    \right)^2\\
    &< \infty,
\end{align}
which yields $u \in \widetilde{L^{\infty}_{x_n}}(\dB_{p,r}^{\frac{n-1}{p}-1})_{x'}(\mathbb{R}^n_+)$.
Hence, we complete the proof.
\end{proof}

\begin{proof}[Proof of Theorem \ref{thm:2}]
Let $0< \delta_1 \leq \delta_0$ and $0<\varepsilon_1\leq \varepsilon_0$ be a constant so small that for any $\bar{F} \in \dB_{p,r}^{\frac{n-1}{p}-2}(\mathbb{R}^{n-1})$ with $\| \bar{F} \|_{\dB_{p,r}^{\frac{n-1}{p}-2}} \leq \delta_1$, \eqref{eq:NS_n-1} possesses a unique solution $\bar{u} \in \dB_{p,r}^{\frac{n-1}{p}-1}(\mathbb{R}^{n-1})$ satisfying $\| \bar{u} \|_{\dB_{p,r}^{\frac{n-1}{p}-1}} \leq \varepsilon_1$.
Let $u$ be the solution to \eqref{eq:NS} constructed in Theorem \ref{thm:1}.
Since $(\bar{u},\bar{p})$ satisfy
\begin{align}
    \begin{cases}
        -\Delta \bar{u} + (\bar{u} \cdot \nabla) \bar{u} + \nabla \bar{p} = \div \bar{F}, \qquad &x \in\mathbb{R}^n_+,\\
        \div \bar{u} = 0 & x \in \mathbb{R}^n_+,
    \end{cases}
\end{align}
the perturbations
$U := u - \bar{u}$ and $P := p - \bar{p}$ solve
\begin{align}
    \begin{cases}
    -\Delta U + (u \cdot \nabla)U + (U \cdot \nabla)\bar{u} + \nabla P = 0, \qquad & x \in \mathbb{R}^n_+,\\
    \div U = 0, & x \in \mathbb{R}^n_+,\\
    U(x',0)=a(x') - \bar{u}(x')=:b(x'), & x' \in \mathbb{R}^{n-1}
    \end{cases}
\end{align}
and thus the corresponding integral equation is given by 
\begin{align}
    U = \mathcal{U}^{\rm boundary}[b] - \mathcal{U}^{\rm force}[ u \otimes U +  U \otimes \bar{u}]. 
\end{align}
Let $\widetilde{R}>R>0$ and $x_n>\widetilde{R}$.
Then, we see from \eqref{est-lin-1} that 
\begin{align}
    \n{\Delta_jU(\cdot,x_n)}_{L^p}
    \leq{}&
    C
    e^{-c2^jx_n}\n{\Delta_jb}_{L^p}\\
    &
    +
    C
    \int_{0}^{\infty}
    e^{-c2^j|x_n-y_n|}
    \n{\Delta_j(u\otimes U + U \otimes \bar{u})(\cdot,y_n)}_{L^p}
    dy_n\\
    \leq{}&
    C
    e^{-c2^j\widetilde{R}}
    \n{\Delta_jb}_{L^p}\\
    &
    +
    C
    \int_{0}^{R}
    e^{-c2^j(x_n-y_n)}
    dy_n
    \n{\Delta_j(u\otimes U + U \otimes \bar{u})}_{L^{\infty}_{x_n}(0,\infty;L^p_{x'})}
    \\
    &
    +
    C
    \int_{R}^{\infty}
    e^{-c2^j|x_n-y_n|}
    dy_n
    \n{\Delta_j(u\otimes U + U \otimes \bar{u})}_{L^{\infty}_{x_n}(R,\infty;L^p_{x'})}
    \\
    \leq{}&
    C
    e^{-c2^j\widetilde{R}}
    \n{\Delta_jb}_{L^p}\\
    &
    +
    C
    e^{-c2^j(\widetilde{R}-R)}
    2^{-j}
    \n{\Delta_j(u\otimes U + U \otimes \bar{u})}_{L^{\infty}_{x_n}(0,\infty;L^p_{x'})}
    \\
    &
    +
    C2^{-j}
    \n{\Delta_j(u\otimes U + U \otimes \bar{u})}_{L^{\infty}_{x_n}(R,\infty;L^p_{x'})}.
\end{align}
Taking $L^{\infty}_{x_n}(\widetilde{R},\infty)$-norm and then $\ell^r(\mathbb{Z})$-norm for $j$ with the weight $2^{(\frac{n-1}{p}-1)j}$, we have 
\begin{align}
    &\n{U}_{\widetilde{L^{\infty}_{x_n}}(\widetilde{R};\infty;\dB_{p,r}^{\frac{n-1}{p}-1})}
    \leq{}
    C
    \left\{
    \sum_{j\in \mathbb{Z}}
    \left(
    e^{-c2^j\widetilde{R}}
    2^{(\frac{n-1}{p}-1)j}
    \n{\Delta_jb}_{L^p}
    \right)^r
    \right\}^{\frac{1}{r}}\\
    &\quad
    +
    C
    \left\{
    \sum_{j\in \mathbb{Z}}
    \left(
    e^{-c2^j(\widetilde{R}-R)}
    2^{(\frac{n-1}{p}-2)j}
    \n{\Delta_j(u\otimes U + U \otimes \bar{u})}_{L^{\infty}_{x_n}(0,\infty;L^p_{x'})}
    \right)^r
    \right\}^{\frac{1}{r}}\\
    &\quad
    +
    C
    \left(
    \n{u}_{\widetilde{L^{\infty}_{x_n}}(\dB_{p,r}^{\frac{n-1}{p}-1})_{x'}}
    +
    \n{\bar{u}}_{\dB_{p,r}^{\frac{n-1}{p}-1}}
    \right)
    \n{U}_{\widetilde{L^{\infty}_{x_n}}(R,\infty;(\dB_{p,r}^{\frac{n-1}{p}-1})_{x'})}.
\end{align}
Letting $\widetilde{R} \to \infty$ via the dominated convergence theorem, we see that 
\begin{align}
    &
    \limsup_{\widetilde{R} \to \infty}
    \n{U}_{\widetilde{L^{\infty}_{x_n}}(\widetilde{R};\infty;\dB_{p,r}^{\frac{n-1}{p}-1})}\\
    &
    \quad
    \leq{}
    C
    \left(
    \n{u}_{\widetilde{L^{\infty}_{x_n}}(\dB_{p,r}^{\frac{n-1}{p}-1})_{x'}}
    +
    \n{\bar{u}}_{\dB_{p,r}^{\frac{n-1}{p}-1}}
    \right)
    \n{U}_{\widetilde{L^{\infty}_{x_n}}(R,\infty;(\dB_{p,r}^{\frac{n-1}{p}-1})_{x'})}.
\end{align}
Hence, taking the limit $R \to \infty$ and using the smallness conditions on $u$ and $\bar{u}$, we complete the proof.
\end{proof}

\noindent
{\bf Data availability.}\\
Data sharing not applicable to this article as no datasets were generated or analysed during the current study.

\noindent
{\bf Conflict of interest.}\\
The author has declared no conflicts of interest.

\noindent
{\bf Acknowledgements.} \\
The author was supported by Grant-in-Aid for Research Activity Start-up, Grant Number JP23K19011.
The author would like to express his sincere gratitude to Professor Yang Li, School of Mathematical Sciences and Center of Pure Mathematics, Anhui University, for many fruitful comments on this manuscript.
The author also thanks the anonymous referees for helpful suggestions and comments.

\end{document}